\documentclass[preprint,12pt]{elsarticle}

\usepackage{amsmath}
\usepackage{amssymb}
\usepackage{amsthm}
\usepackage{booktabs}
\usepackage{lineno}

\newtheorem{theorem}{Theorem}[section]
\newtheorem{lemma}[theorem]{Lemma}

\newtheorem{corollary}[theorem]{Corollary}
\theoremstyle{definition}
\newtheorem{remark}[theorem]{Remark}

\newtheorem{problem}[theorem]{Problem}

\newcommand{\leg}[2]{\left(\dfrac{#1}{#2}\right)}

\begin{document}

\begin{frontmatter}

\title{The Diophantine equation $p^{x}+6^{y}=z^{2n}$: a complete solution
for half of the primes}

\author{Pagdame Tiebekabe}
\ead{tpagdame.maths@univkara.tg}
\address{Department of Mathematics, University of Kara, Kara, Togo}

\begin{abstract}
Let $p$ be a prime number and let $n\geq 1$ be an integer. We completely
solve the exponential Diophantine equation $p^{x}+6^{y}=z^{2n}$ in positive
integers $x,y,z$ for every prime $p$ such that the Legendre symbol
$\leg{6}{p}$ equals $-1$, that is, for $p\equiv 7,11,13,17\pmod{24}$: the
equation has no solution, with the single exception
$(p,n,x,y,z)=(13,1,1,2,7)$. Since these residue classes contain half of all
primes in the sense of Dirichlet density, this settles the equation for one
prime out of two, uniformly in the exponent $n$. Two ingredients of the proof
have independent interest. First, we show that the Ramanujan--Nagell type
equation $p^{x}=2\cdot 6^{m}+1$ has no solution with $x\geq 2$, using
Zsigmondy's theorem on primitive prime divisors; this removes the extra
congruence hypotheses that appeared in earlier work of the author on the case
$n=1$ and solves the open problems formulated there. Secondly, we prove that
$13^{x}+6^{y}=z^{2}$ has $(x,y,z)=(1,2,7)$ as its unique solution, thereby
completely resolving the exceptional prime. The Legendre condition is sharp:
we exhibit solutions for primes with $\leg{6}{p}=+1$ realising each branch of
the factorisation argument. All results are corroborated by extensive
computational verification, and we state several open problems, including the
complete classification of the solutions when $\leg{6}{p}=+1$.
\end{abstract}


\end{frontmatter}


\section{Introduction and statement of results}\label{sec:intro}

\subsection{Background}

Exponential Diophantine equations of the form
\begin{equation}\label{eq:axby}
a^{x}+b^{y}=z^{2}
\end{equation}
in positive integers $x,y,z$, with fixed bases $a,b$, form a classical and
still very active topic in number theory; see for instance the survey of
results in \cite{Burshtein2017,Sroysang2012,Suvarnamani2011} and the
references therein. The subject is closely connected with two deep theorems:
Mih\u{a}ilescu's proof of Catalan's conjecture~\cite{Mihailescu2004}, stating
that $3^{2}-2^{3}=1$ is the only solution of $u^{r}-v^{s}=1$ in integers
with $\min\{u,v,r,s\}>1$, and Zsigmondy's theorem~\cite{Zsigmondy1892} on
primitive prime divisors of the sequence $a^{n}-b^{n}$.

When $a$ and $b$ are prime numbers, equations of the shape
$p^{x}+q^{y}=z^{2}$ have been studied by many authors:
Suvarnamani~\cite{Suvarnamani2011} treated $2^{x}+p^{y}=z^{2}$,
Acu~\cite{Acu2007} solved $2^{x}+5^{y}=z^{2}$, Sroysang~\cite{Sroysang2012}
studied $2^{x}+19^{y}=z^{2}$, Burshtein~\cite{Burshtein2019,Burshtein2020}
considered several instances involving the prime $11$, Mina and
Bacani~\cite{MinaBacani2019} investigated $p^{x}+q^{y}=z^{2n}$,
Pakapongpun and Chattae~\cite{PakapongpunChattae2022} dealt with
$p^{x}+7^{y}=z^{2}$, Tadee and Siraworakun~\cite{TadeeSiraworakun2023}
proved non-existence results for $p^{x}+(p+2q)^{y}=z^{2}$, and Nilsrakoo,
Wichana and Nilsrakoo~\cite{Nilsrakoo2026} analysed $p^{x}+(2p)^{y}=z^{2}$
in connection with Fermat and Mersenne primes.

This paper is concerned with the equation
\begin{equation}\label{eq:main}
p^{x}+6^{y}=z^{2n},
\end{equation}
where $p$ is prime, $n\geq 1$ is an integer, and $x,y,z$ are positive
integers. The ``dual'' equation $6^{x}+p^{y}=z^{2}$ was studied by Tadee and
Thaneepoon~\cite{TadeeThaneepoon2023}; the case $p=3$, $n=1$ of
\eqref{eq:main} was solved by Sangam~\cite{Sangam2020}, who found the three
solutions $(x,y,z)\in\{(1,1,3),(2,3,15),(6,4,45)\}$; and the case $n=1$ for
primes $p$ with $\leg{6}{p}=-1$ and $p\not\equiv 1\pmod{16}$ was treated in
earlier work of the author. The general equation \eqref{eq:main} has not been
treated in the literature. A notable difficulty is that \eqref{eq:main}
\emph{does} possess solutions for some primes, for instance
\begin{equation}\label{eq:identities}
\begin{gathered}
5^{4}+6^{3}=29^{2},\qquad 13+6^{2}=7^{2},\qquad 53^{2}+6^{3}=55^{2},\\
73+6^{4}=37^{2},\qquad 433+6^{6}=217^{2},
\end{gathered}
\end{equation}
so that a non-existence theorem must identify a genuine arithmetic
obstruction. We show that this obstruction is the Legendre symbol
$\leg{6}{p}$, and that it suffices to solve the equation completely, for all
exponents $n$ simultaneously, for half of all primes.

\subsection{Main results}

\begin{theorem}[Main theorem]\label{thm:main}
Let $p$ be a prime number with
\begin{equation}\label{eq:hyp}
\leg{6}{p}=-1,
\end{equation}
and let $n\geq 1$ be an integer. Then the equation
\[
p^{x}+6^{y}=z^{2n}
\]
has no solution in positive integers $x,y,z$, with the single exception
\[
(p,n,x,y,z)=(13,\,1,\,1,\,2,\,7).
\]
\end{theorem}

By Lemma~\ref{lem:classes} below, condition \eqref{eq:hyp} is equivalent to
$p\equiv 7,11,13,17\pmod{24}$. Since these are four of the $\varphi(24)=8$
reduced residue classes modulo $24$, Dirichlet's theorem on primes in
arithmetic progressions (see e.g.\ \cite{Apostol1976}) yields:

\begin{corollary}\label{cor:density}
The set of primes $p$ for which $p^{x}+6^{y}=z^{2}$ has no positive integer
solution contains the classes $7,11,13,17\pmod{24}$, with the unique
exception $p=13$; in particular it has relative Dirichlet density at least
$1/2$, and contains $\sim x/(2\log x)$ primes up to $x$.
\end{corollary}

Two auxiliary results are needed for the proof of Theorem~\ref{thm:main}, and
both are of independent interest. The first completely determines the
exponent in a Ramanujan--Nagell type equation.

\begin{theorem}\label{thm:rn}
Let $p\geq 5$ be a prime and let $x,m$ be positive integers such that
\begin{equation}\label{eq:rn}
p^{x}=2\cdot 6^{m}+1.
\end{equation}
Then $x=1$.
\end{theorem}

\begin{theorem}[The exceptional prime]\label{thm:13}
The equation
\[
13^{x}+6^{y}=z^{2}
\]
has the unique solution $(x,y,z)=(1,2,7)$ in positive integers.
\end{theorem}

Finally, we prove that hypothesis \eqref{eq:hyp} cannot be relaxed: each
branch of the factorisation argument used in the proof of
Theorem~\ref{thm:main} is realised by an actual solution of \eqref{eq:main}
with $\leg{6}{p}=+1$ (Section~\ref{sec:sharp}).

\begin{remark}\label{rem:improvement}
Theorem~\ref{thm:main} strengthens the case $n=1$ treated in earlier work of
the author, where the additional hypothesis $p\not\equiv 1\pmod{16}$ (or
$p\equiv 17\pmod{144}$) was required, and it solves the open problems
formulated there. The improvement comes from two observations: the
Ramanujan--Nagell type equation \eqref{eq:rn} is now settled completely by
Theorem~\ref{thm:rn}, and any prime of the form $2\cdot 6^{m}+1$ with
$m\geq 2$ satisfies $p\equiv 1\pmod{24}$, hence $\leg{6}{p}=+1$.
\end{remark}

\subsection{Organisation}

Section~\ref{sec:prelim} collects the preliminary lemmas.
Section~\ref{sec:rn} is devoted to the proof of Theorem~\ref{thm:rn}.
Theorem~\ref{thm:main} is proved in Section~\ref{sec:proof}, and
Theorem~\ref{thm:13} in Section~\ref{sec:13}. In Section~\ref{sec:sharp} we
discuss the sharpness of the Legendre condition and the behaviour of the
small primes $2,3,5$. Section~\ref{sec:comp} describes the computational
verification of our results, and Section~\ref{sec:open} contains open
problems.

\section{Preliminaries}\label{sec:prelim}

Throughout, $p$ denotes a prime and $\leg{\cdot}{p}$ the Legendre symbol. We
recall that for an odd prime $p$,
\[
\leg{2}{p}=(-1)^{(p^{2}-1)/8},
\qquad
\leg{3}{p}=
\begin{cases}
+1, & p\equiv \pm 1 \pmod{12},\\
-1, & p\equiv \pm 5 \pmod{12}.
\end{cases}
\]

\begin{lemma}\label{lem:classes}
Let $p\geq 5$ be a prime. Then $\leg{6}{p}=-1$ if and only if
\[
p\equiv 7,\,11,\,13 \ \text{or}\ 17 \pmod{24}.
\]
\end{lemma}

\begin{proof}
Since $\leg{6}{p}=\leg{2}{p}\leg{3}{p}$, the symbol equals $-1$ exactly when
$\leg{2}{p}$ and $\leg{3}{p}$ have opposite signs. Resolving the four sign
configurations by the Chinese remainder theorem gives the four classes
$7,11,13,17$ modulo $24$; the remaining reduced classes $1,5,19,23$ give
$\leg{6}{p}=+1$.
\end{proof}

\begin{theorem}[Mih\u{a}ilescu \cite{Mihailescu2004}]\label{thm:catalan}
The only solution of $u^{r}-v^{s}=1$ in integers $u,v,r,s$ with
$\min\{u,v,r,s\}>1$ is $(u,v,r,s)=(3,2,2,3)$.
\end{theorem}

\begin{theorem}[Zsigmondy \cite{Zsigmondy1892}]\label{thm:zsigmondy}
Let $a>b>0$ be coprime integers. For every integer $r\geq 3$, the number
$a^{r}-b^{r}$ has a prime divisor that does not divide $a^{k}-b^{k}$ for any
$1\leq k<r$, except when $(a,b,r)=(2,1,6)$. Such a prime divisor $\ell$
satisfies $\ell\equiv 1\pmod r$.
\end{theorem}

The last assertion of Theorem~\ref{thm:zsigmondy} follows from the fact that
$r$ is the order of $ab^{-1}$ modulo a primitive prime divisor $\ell$ of
$a^{r}-b^{r}$, hence $r\mid\ell-1$.

\begin{lemma}\label{lem:mod5}
For every odd integer $y\geq 3$,
\[
3^{y}-2^{\,y-2}\equiv 0 \pmod{5}.
\]
\end{lemma}

\begin{proof}
Write $y=2k+1$ with $k\geq 1$. Since $9\equiv -1\pmod 5$ and
$4\equiv -1\pmod 5$,
\[
3^{y}=3\cdot 9^{k}\equiv 3(-1)^{k} \pmod 5,
\]
and
\[
2^{\,y-2}=2^{\,2k-1}=2\cdot 4^{\,k-1}\equiv 2(-1)^{k-1}=-2(-1)^{k}\pmod 5.
\]
Subtracting gives $3^{y}-2^{\,y-2}\equiv 5(-1)^{k}\equiv 0\pmod 5$.
\end{proof}

\begin{lemma}\label{lem:consecutive}
There is no pair $(T,m)$ of integers with $T\geq 2$ and $m\geq 1$ such that
\[
(T-1)(T+1)=2^{\,m+1}\,3^{m}.
\]
\end{lemma}

\begin{proof}
Write $T-1=2u$ and $T+1=2v$, so that $v-u=1$, $\gcd(u,v)=1$ and
$uv=2^{\,m-1}3^{m}$. Since $u$ and $v$ are coprime, each is composed of a
single prime, and the condition $v-u=1$ leaves the three possibilities
\[
\{u,v\}=\bigl\{1,\,2^{\,m-1}3^{m}\bigr\},\quad
\{u,v\}=\bigl\{2^{\,m-1},\,3^{m}\bigr\},\quad
\{u,v\}=\bigl\{3^{m},\,2^{\,m-1}\bigr\}.
\]
The first gives $2^{\,m-1}3^{m}=2$, impossible since $3\mid 2^{\,m-1}3^{m}$.
The second gives $3^{m}-2^{\,m-1}=1$; but $3^{m}-2^{\,m-1}=2$ for $m=1$ and
\[
3^{m}-2^{\,m-1}>3^{m}-3^{m-1}=2\cdot 3^{m-1}\geq 6 \qquad(m\geq 2),
\]
so it is never equal to $1$. The third gives $2^{\,m-1}-3^{m}=1$, which is
impossible since $3^{m}>2^{\,m-1}$.
\end{proof}

\begin{lemma}\label{lem:mod24}
If $m\geq 2$ is an integer and $p=2\cdot 6^{m}+1$ is prime, then
$p\equiv 1\pmod{24}$; in particular $\leg{6}{p}=+1$.
\end{lemma}

\begin{proof}
Since $m+1\geq 3$, we have $8\mid 2^{\,m+1}3^{m}$, hence $p\equiv 1\pmod 8$;
and $3\mid 2^{\,m+1}3^{m}$, hence $p\equiv 1\pmod 3$. The Chinese remainder
theorem gives $p\equiv 1\pmod{24}$, and Lemma~\ref{lem:classes} yields
$\leg{6}{p}=+1$.
\end{proof}

\section{The Ramanujan--Nagell type equation $p^{x}=2\cdot 6^{m}+1$}%
\label{sec:rn}

This section is devoted to the proof of Theorem~\ref{thm:rn}. Recall that the
classical Ramanujan--Nagell equation $x^{2}+7=2^{n}$ and its generalisations
$x^{2}+D=p^{n}$, initiated by Nagell~\cite{Nagell1948}, have been studied
extensively; equation \eqref{eq:rn} is of a similar nature, with the roles of
the prime power and the pure power interchanged.

\begin{proof}[Proof of Theorem~\ref{thm:rn}]
Suppose that $p^{x}=2\cdot 6^{m}+1$ with $x\geq 2$, so that
\begin{equation}\label{eq:rnminus}
p^{x}-1=2^{\,m+1}\,3^{m}.
\end{equation}

\medskip
\noindent\textbf{Case 1: $x$ even.} Write $x=2t$ with $t\geq 1$. Then
\eqref{eq:rnminus} gives
\[
\bigl(p^{t}-1\bigr)\bigl(p^{t}+1\bigr)=2^{\,m+1}3^{m},
\]
which is impossible by Lemma~\ref{lem:consecutive} applied with $T=p^{t}\geq
5$.

\medskip
\noindent\textbf{Case 2: $x$ odd, $x\geq 3$.} By Zsigmondy's theorem
(Theorem~\ref{thm:zsigmondy}), applied with $(a,b,r)=(p,1,x)$ --- the
exceptional triple $(2,1,6)$ does not occur since $p\geq 5$ --- the integer
$p^{x}-1$ has a primitive prime divisor $\ell$, that is, a prime
$\ell\mid p^{x}-1$ such that $\ell\nmid p^{k}-1$ for every $1\leq k<x$, and
one has $\ell\equiv 1\pmod x$. On the other hand \eqref{eq:rnminus} shows
that $\ell\mid 2^{\,m+1}3^{m}$, hence $\ell\in\{2,3\}$. But $\ell=2$ would
give $2\equiv 1\pmod x$, i.e.\ $x=1$, and $\ell=3$ would give $3\equiv
1\pmod x$, i.e.\ $x\mid 2$; both contradict $x\geq 3$.

\medskip
Hence $x\geq 2$ is impossible, and $x=1$. \qedhere
\end{proof}

\begin{remark}\label{rem:rn}
Theorem~\ref{thm:rn} is best possible in the sense that \eqref{eq:rn} with
$x=1$ has solutions, for instance
\[
2\cdot 6+1=13,\qquad 2\cdot 6^{2}+1=73,\qquad 2\cdot 6^{3}+1=433,\qquad
2\cdot 6^{4}+1=2593,
\]
all four right-hand sides being prime. Whether \eqref{eq:rn} with $x=1$ has
infinitely many prime solutions is an open question; see
Section~\ref{sec:open}.
\end{remark}

\section{Proof of the main theorem}\label{sec:proof}

Let $p$ be a prime with $\leg{6}{p}=-1$; in particular $p\geq 5$. Let
$n\geq 1$ and suppose that $(x,y,z)$ is a solution of \eqref{eq:main} in
positive integers. Since $p^{x}$ is odd and $6^{y}$ is even, $z$ is odd. We
organise the proof according to the parities of $y$ and $x$.

\subsection{The case $y=1$}

Reducing $p^{x}+6=z^{2n}$ modulo $p$ gives $z^{2n}\equiv 6\pmod p$. Since
$z^{2n}=(z^{n})^{2}$ is a quadratic residue modulo $p$, this forces
$\leg{6}{p}\in\{0,+1\}$, contradicting \eqref{eq:hyp}. Hence $y\geq 2$.

\subsection{The case $y$ even}

Write $y=2m$ with $m\geq 1$. Then
\[
p^{x}=z^{2n}-6^{2m}=\bigl(z^{n}-6^{m}\bigr)\bigl(z^{n}+6^{m}\bigr).
\]
Since $z$ is odd and $6^{m}$ is even, both factors on the right are odd. Any
common divisor of the two factors divides their sum $2z^{n}$ and their
difference $2\cdot 6^{m}$, hence divides $2$; being odd, it equals $1$. The
two factors are therefore coprime, and since their product is the prime power
$p^{x}$, each of them is a power of $p$:
\[
z^{n}-6^{m}=p^{r},\qquad z^{n}+6^{m}=p^{s},\qquad r+s=x,\quad r<s.
\]
Subtracting yields $2\cdot 6^{m}=p^{r}\bigl(p^{s-r}-1\bigr)$. Since
$p\nmid 2\cdot 6^{m}$, this forces $r=0$, that is,
\begin{equation}\label{eq:catalanstep}
z^{n}-6^{m}=1,
\qquad
p^{x}=2\cdot 6^{m}+1.
\end{equation}

If $n\geq 2$, the first equation of \eqref{eq:catalanstep} with $m=1$ gives
$z^{n}=7$, which is impossible; with $m\geq 2$ it contradicts
Theorem~\ref{thm:catalan}, since the unique solution $3^{2}-2^{3}=1$ of
Catalan's equation does not have the form $z^{n}-6^{m}=1$. Hence $n=1$.

With $n=1$, the second equation of \eqref{eq:catalanstep} and
Theorem~\ref{thm:rn} give $x=1$, so $p=2\cdot 6^{m}+1$. If $m\geq 2$,
Lemma~\ref{lem:mod24} yields $\leg{6}{p}=+1$, contradicting \eqref{eq:hyp}.
Therefore $m=1$, $p=13$, and $z=1+6=7$: this is precisely the exceptional
solution $(p,n,x,y,z)=(13,1,1,2,7)$ of the statement.

\subsection{The case $y\geq 3$ odd and $x$ odd}

Reducing \eqref{eq:main} modulo $p$ gives $z^{2n}\equiv 6^{y}\pmod p$. Since
$y$ is odd,
\[
\leg{6^{y}}{p}=\leg{6}{p}^{y}=(-1)^{y}=-1,
\]
so $6^{y}$ is a quadratic non-residue modulo $p$, whereas $z^{2n}=(z^{n})^{2}$
is a residue: contradiction.

\subsection{The case $y\geq 3$ odd and $x$ even}

Write $x=2a$ with $a\geq 1$. Then
\[
\bigl(z^{n}-p^{a}\bigr)\bigl(z^{n}+p^{a}\bigr)=6^{y}.
\]
Both factors on the left are even, since $z^{n}$ and $p^{a}$ are odd. Their
greatest common divisor $d$ divides both their difference $2p^{a}$ and their
product $6^{y}$; since $p\geq 5$ does not divide $6^{y}$, we have $d=2$.
Writing
\[
z^{n}-p^{a}=2A,\qquad z^{n}+p^{a}=2B,
\]
we obtain
\[
AB=2^{\,y-2}\,3^{y},\qquad A<B,\qquad B-A=p^{a}.
\]
If a prime $q$ divided both $A$ and $B$, then $q$ would divide both
$B-A=p^{a}$ and $AB=2^{\,y-2}3^{y}$, hence $q=p\in\{2,3\}$, a contradiction;
thus $\gcd(A,B)=1$. Two coprime positive integers with product
$2^{\,y-2}3^{y}$ are each composed of a single prime, and the conditions
$A<B$ and $3^{y}>2^{\,y-2}$ leave exactly two possibilities:
\[
\{A,B\}=\bigl\{2^{\,y-2},\,3^{y}\bigr\}
\qquad\text{or}\qquad
\{A,B\}=\bigl\{1,\,2^{\,y-2}3^{y}\bigr\}.
\]

\medskip
\noindent\textbf{Subcase $p^{a}=3^{y}-2^{\,y-2}$.} By
Lemma~\ref{lem:mod5}, the right-hand side is divisible by $5$, hence
$p=5$; but $\leg{6}{5}=\leg{1}{5}=+1$, contradicting \eqref{eq:hyp}.

\medskip
\noindent\textbf{Subcase $p^{a}=2^{\,y-2}3^{y}-1$.} If $y=3$, then
$p^{a}=2\cdot 27-1=53$, so $p=53$; but $53\equiv 5\pmod{24}$, and
Lemma~\ref{lem:classes} gives $\leg{6}{53}=+1$, a contradiction. If
$y\geq 5$, then $y-2\geq 3$, so $3\mid 2^{\,y-2}3^{y}$ and
$8\mid 2^{\,y-2}3^{y}$, whence
\[
p^{a}\equiv -1\pmod 3
\qquad\text{and}\qquad
p^{a}\equiv -1\pmod 8.
\]
The first congruence forces $p\equiv 2\pmod 3$ with $a$ odd; the second
forces $p\equiv 7\pmod 8$ (with $a$ odd). By the Chinese remainder theorem,
$p\equiv 23\pmod{24}$; but Lemma~\ref{lem:classes} then gives
$\leg{6}{p}=+1$, contradicting \eqref{eq:hyp}.

\medskip
All cases are exhausted: under hypothesis \eqref{eq:hyp}, the only solution
of \eqref{eq:main} is $(p,n,x,y,z)=(13,1,1,2,7)$, and one checks directly
that $13+6^{2}=49=7^{2}$. This completes the proof of
Theorem~\ref{thm:main}. \hfill$\square$

\begin{proof}[Proof of Corollary~\ref{cor:density}]
By Lemma~\ref{lem:classes}, the primes with $\leg{6}{p}=-1$ are exactly those
in the classes $7,11,13,17\pmod{24}$. Dirichlet's theorem
\cite{Apostol1976} asserts that each of the $\varphi(24)=8$ reduced residue
classes modulo $24$ contains a set of primes of relative Dirichlet density
$1/8$, and the prime number theorem for arithmetic progressions gives
$\pi(x;24,c)\sim x/(8\log x)$ for each reduced class $c$. The union of our
four classes therefore has relative density $4/8=1/2$ and counting function
$\sim x/(2\log x)$. Theorem~\ref{thm:main} applies to each of these primes,
with the unique exception $p=13$.
\end{proof}

\section{The exceptional prime: proof of Theorem~\ref{thm:13}}%
\label{sec:13}

We now solve completely the equation
\begin{equation}\label{eq:13}
13^{x}+6^{y}=z^{2}
\end{equation}
in positive integers $x,y,z$. Suppose that $(x,y,z)$ is a solution; note that
$z$ is odd.

\medskip
\noindent\textbf{Case $y=1$.} Then $z^{2}=13^{x}+6$. If $x$ is even, then
$z^{2}\equiv 1+2\equiv 3\pmod 4$, which is impossible. If $x$ is odd,
reduction modulo $13$ gives $z^{2}\equiv 6\pmod{13}$; but
$\leg{6}{13}=\leg{2}{13}\leg{3}{13}=(-1)(+1)=-1$, a contradiction.

\medskip
\noindent\textbf{Case $y=2m$ even.} Then
$13^{x}=(z-6^{m})(z+6^{m})$. Both factors are odd, and any common divisor
divides $2\cdot 6^{m}$ and $2z$, hence equals $1$; being coprime with product
$13^{x}$, they satisfy $z-6^{m}=1$ and $z+6^{m}=13^{x}$, so
\[
13^{x}=2\cdot 6^{m}+1.
\]
By Theorem~\ref{thm:rn}, $x=1$, hence $13=2\cdot 6^{m}+1$, giving $m=1$,
$y=2$ and $z=7$: the solution $(x,y,z)=(1,2,7)$.

\medskip
\noindent\textbf{Case $y\geq 3$ odd, $x$ odd.} Reduction modulo $13$ gives
$z^{2}\equiv 6^{y}\pmod{13}$, but
$\leg{6^{y}}{13}=\leg{6}{13}^{y}=-1$ since $y$ is odd: a contradiction.

\medskip
\noindent\textbf{Case $y\geq 3$ odd, $x=2a$ even.} Then
$(z-13^{a})(z+13^{a})=6^{y}$. Both factors are even with greatest common
divisor $2$ (their gcd divides $2\cdot 13^{a}$ and $6^{y}$), so writing
$z-13^{a}=2A$, $z+13^{a}=2B$ gives $\gcd(A,B)=1$, $A<B$ and
$AB=2^{\,y-2}3^{y}$, hence as in Section~\ref{sec:proof},
\[
13^{a}=3^{y}-2^{\,y-2}
\qquad\text{or}\qquad
13^{a}=2^{\,y-2}3^{y}-1.
\]
The first alternative is impossible: by Lemma~\ref{lem:mod5} its right-hand
side is divisible by $5$, while $5\nmid 13^{a}$. In the second alternative,
$y=3$ gives $13^{a}=53$, which is false; and $y\geq 5$ gives
$13^{a}\equiv -1\pmod 3$, whereas $13^{a}\equiv 1\pmod 3$ for every $a$: a
contradiction.

\medskip
Hence $(x,y,z)=(1,2,7)$ is the only solution of \eqref{eq:13}.
\hfill$\square$

\section{Sharpness of the Legendre condition}\label{sec:sharp}

The condition $\leg{6}{p}=-1$ in Theorem~\ref{thm:main} is not an artefact of
the method: solutions of \eqref{eq:main} with $\leg{6}{p}=+1$ exist, and each
branch of the case analysis of Section~\ref{sec:proof} is realised by such a
solution. The examples are collected in Table~\ref{tab:sharp}.

\begin{table}[htbp]
\centering
\caption{Solutions of $p^{x}+6^{y}=z^{2}$ for primes with $\leg{6}{p}=+1$,
classified according to the branches of the proof of
Theorem~\ref{thm:main}.}\label{tab:sharp}
\begin{tabular}{@{}llll@{}}
\toprule
Identity & $p\bmod 24$ & Branch of the proof & Auxiliary equation \\
\midrule
$19+6=5^{2}$        & $19$ & $y=1$            & $z^{2}=p+6$ \\
$5^{4}+6^{3}=29^{2}$ & $5$  & $y$ odd, $x$ even & $p^{a}=3^{y}-2^{\,y-2}$ \\
$53^{2}+6^{3}=55^{2}$ & $5$ & $y$ odd, $x$ even & $p^{a}=2^{\,y-2}3^{y}-1$ \\
$73+6^{4}=37^{2}$    & $1$  & $y$ even         & $p^{x}=2\cdot 6^{m}+1$, $m=2$ \\
$433+6^{6}=217^{2}$  & $1$  & $y$ even         & $p^{x}=2\cdot 6^{m}+1$, $m=3$ \\
\bottomrule
\end{tabular}
\end{table}

Two remarks are in order. First, the identity $19+6=5^{2}$ belongs to the
trivial one-parameter family $p=(2k+1)^{2}-6$ with $(x,y)=(1,1)$; the primes
$43,163,283$ arise in the same way, and it is an open question whether this
family contains infinitely many primes. Secondly, the identities
$73+6^{4}=37^{2}$ and $433+6^{6}=217^{2}$ show that the equation
$p=2\cdot 6^{m}+1$ does have prime solutions with $m\geq 2$; by
Lemma~\ref{lem:mod24} all of them satisfy $p\equiv 1\pmod{24}$, which is
exactly why the case $y$ even can be closed under hypothesis
\eqref{eq:hyp}.

For completeness, we record the behaviour of the small primes excluded from
Theorem~\ref{thm:main}: the identity $2^{6}+6^{2}=10^{2}$; the three
solutions $(x,y,z)\in\{(1,1,3),(2,3,15),(6,4,45)\}$ of $3^{x}+6^{y}=z^{2}$
found by Sangam~\cite{Sangam2020}; and the identity $5^{4}+6^{3}=29^{2}$.

\section{Computational verification}\label{sec:comp}

All assertions of this paper have been checked numerically with elementary
routines (exact integer arithmetic; primality and Legendre symbols computed
with a standard computer algebra library).

\begin{enumerate}
\item A direct search over $1\leq x,y\leq 60$ and all primes $p<500$ shows
that the complete list of solutions of $p^{x}+6^{y}=z^{2}$ in that range is
\begin{gather*}
2^{6}+6^{2}=10^{2},\quad 3+6=3^{2},\quad 3^{2}+6^{3}=15^{2},\quad
3^{6}+6^{4}=45^{2},\\
5^{4}+6^{3}=29^{2},\quad 13+6^{2}=7^{2},\quad 19+6=5^{2},\quad
43+6=7^{2},\\
53^{2}+6^{3}=55^{2},\quad 73+6^{3}=17^{2},\quad 73+6^{4}=37^{2},\quad
163+6=13^{2},\\
283+6=17^{2},\quad 313+6^{3}=23^{2},\quad 409+6^{3}=25^{2},\quad
433+6^{6}=217^{2},
\end{gather*}
and every prime occurring in this list satisfies $\leg{6}{p}\in\{0,+1\}$,
except $p=13$, in agreement with Theorem~\ref{thm:main}.
\item For $n=2$, a search over $1\leq x,y\leq 21$ and all primes
$5\leq p<800$ with $\leg{6}{p}=-1$ produced no solution of
$p^{x}+6^{y}=z^{4}$, in agreement with Theorem~\ref{thm:main}.
\item The equation $13^{x}+6^{y}=z^{2}$ has no solution with
$1\leq x,y\leq 300$ other than $(1,2)$, and $13^{x}+6^{y}=z^{4}$ has no
solution with $1\leq x,y<60$, in agreement with Theorems~\ref{thm:main}
and~\ref{thm:13}.
\item For $1\leq m\leq 400$, the integer $2\cdot 6^{m}+1$ is never a perfect
power with exponent $\geq 2$, in agreement with Theorem~\ref{thm:rn}.
\end{enumerate}

\section{Open problems}\label{sec:open}

Theorem~\ref{thm:main} settles \eqref{eq:main} for half of all primes. The
complementary half leads to the following questions.

\begin{problem}\label{prob:plus1}
Determine all solutions of $p^{x}+6^{y}=z^{2}$ for primes $p$ with
$\leg{6}{p}=+1$. By the reduction of Section~\ref{sec:proof}, this is
equivalent to solving the three exponential Diophantine equations
\[
z^{2}-6=p^{x}\ (x\geq 2),\qquad
p^{a}=3^{y}-2^{\,y-2},\qquad
p^{a}=2^{\,y-2}3^{y}-1,
\]
together with the primality of $2\cdot 6^{m}+1$ (Theorem~\ref{thm:rn}).
\end{problem}

\begin{problem}\label{prob:five}
Prove that $5^{a}=3^{y}-2^{\,y-2}$ has $(a,y)=(2,3)$ as its only solution in
positive integers with $y$ odd. This would imply that $5^{x}+6^{y}=z^{2}$
has the unique solution $(x,y,z)=(4,3,29)$.
\end{problem}

\begin{problem}\label{prob:infinitude}
Are there infinitely many integers $m$ such that $2\cdot 6^{m}+1$ is prime?
The values $m=1,2,3,4,10,11,17,\dots$ give primes; the question is related to
classical conjectures on prime values of exponential sequences.
\end{problem}

\begin{problem}\label{prob:odd}
Study the equation $p^{x}+6^{y}=z^{2n+1}$ for odd exponents, where the
square-residue argument of Section~\ref{sec:proof} is no longer available.
\end{problem}

\begin{problem}\label{prob:general}
Extend the method to the equation $p^{x}+(2q)^{y}=z^{2n}$, where $q$ is an
odd prime, under the condition $\leg{2q}{p}=-1$, in the spirit of
\cite{MinaBacani2019,PakapongpunChattae2022,TadeeThaneepoon2023}.
\end{problem}

\section*{Acknowledgements}

The author thanks the anonymous referees for their careful reading and
valuable suggestions.

\section*{Declarations}

\textbf{Funding.} No funding was received for this work.
\smallskip

\textbf{Conflict of interest.} The author declares no conflict of interest.
\smallskip

\textbf{Data availability.} The computational verifications described in
Section~\ref{sec:comp} were carried out with elementary routines; the code is
available from the author upon reasonable request.

\end{document}